\documentclass[11pt,twoside,letterpaper]{article} %% The same for {book}
\usepackage{times,fancyhdr}
\usepackage[dvips]{graphicx}

\def\figurename{Figure} % (Replace the colon that normally
\makeatletter
\renewcommand{\fnum@figure}[1]{\figurename~\thefigure.}
\makeatother

\def\tablename{Table} % (Replace the colon that normally
\makeatletter
\renewcommand{\fnum@table}[1]{\tablename~\thetable.}
\makeatother

\usepackage{mathrsfs}
\usepackage{amsmath}
\usepackage{amssymb}
\usepackage{amsfonts}
\usepackage{amsthm,amscd}
\usepackage{xcolor}
\usepackage{anysize,enumerate}

\newtheorem{assumption}{Assumption}[section]
\newtheorem{theorem}{Theorem}[section]
\newtheorem{lemma}[theorem]{Lemma}
\newtheorem{corollary}[theorem]{Corollary}

\theoremstyle{definition}
\newtheorem{definition}[theorem]{Definition}

\theoremstyle{remark}

\numberwithin{equation}{section}

\def\K{\mathcal K}
\def\h{\mathcal H}
\def\b{\mathcal B}
\def\r{\mathbb R}
\def\n{\mathbb N}

\def\lo{\Omega}

\def\s{\sigma}
\def\ls{\Sigma}
\def\d{\mathrm d}
\def\e{\epsilon}

\begin{document}
%\vskip 0.4in

\title{\bfseries  Uniform attractors of  non-autonomous Kirchhoff wave models
  \footnote{Supported by
Natural Science Foundation of China (No.11671367). e-mail: yzjzzut@tom.com, LYN20112110109@163.com, fnadream2011@163.com} }
\author{\bfseries  Zhijian Yang, \ \   Yanan Li, \ \ Na Feng \\
 School of Mathematics and Statistics, Zhengzhou
University, No.100, Science Road, \\
 Zhengzhou  450001, China}
\date{}
\maketitle \thispagestyle{empty} \setcounter{page}{1}

\begin{abstract} \noindent The paper investigates the existence and  upper semicontinuity of uniform attractors of  the perturbed non-autonomous Kirchhoff wave equations with strong damping and supercritical nonlinearity: $u_{tt}-\Delta u_{t}-(1+\epsilon\|\nabla u\|^{2})\Delta u+f(u)=g(x,t)$, where $\epsilon\in [0,1]$ is a perturbed parameter. It shows that when the nonlinearity $f(u)$ is of supercritical growth $p: \frac{N+2}{N-2}=p^*<p<p^{**}=\frac{N+4}{(N-4)^+}$: (i) the related evolution process has a  compact uniform attractor $\mathcal{A}_\ls^\e $ for each $\epsilon\in [0,1]$; (ii) the family of  uniform attractor $\mathcal{A}_\ls^\e $ is upper semicontinuous on the perturbed parameter  $\epsilon$ in the sense of partially strong topology.

\end{abstract}

\vspace{.08in} \noindent \textbf{Keywords}: Non-autonomous Kirchhoff  wave models; perturbed parameter;
 supercritical nonlinearity; uniform attractor; pullback attractor; upper semicontinuity.

\section{Introduction}
In this paper, we are concerned with  the existence and  upper semicontinuity of uniform attractors of  the perturbed non-autonomous Kirchhoff wave equations with strong damping and supercritical nonlinearity:
\begin{align}&
u_{tt}-\Delta u_{t}-(1+\epsilon\|\nabla u\|^{2})\Delta u+f(u)=g(x,t),\ \ x\in\Omega,\  t>\tau,\label{1.1}\\
&u|_{\partial\Omega}=0,\ \
u(x,\tau)=u^{\tau}_{0},\ \ u_{t}(x,\tau)=u^{\tau}_{1},\ \ \tau\in\mathbb{R},\label{1.2}
\end{align}
where $\lo$ is a bounded domain in $\r^N$ ($N\geq 3$) with the smooth boundary $\partial \lo, \epsilon\in [0,1]$ is a perturbed parameter.
Throughout this paper we use the following notations:
\begin{align*}&
  L^p=L^p(\lo),\ \ H^k=W^{k,2}(\lo),\ \ H_0^k=W_0^{k,2}(\lo),\ \  \|\cdot \|_p=\|\cdot\|_{L^p}, \ \
\|\cdot\|= \|\cdot\|_{L^2},
\end{align*}
 with $p\geq 1.$  The sign $ H_1\hookrightarrow H_2$ denotes that the space $H_1$ continuously  embeds   into $H_2$ and $H_1\hookrightarrow\hookrightarrow H_2$ denotes that  $H_1$ compactly embeds into $H_2$. We denote the phase spaces
\[ \h=(H^1_0\cap L^{p+1})\times L^2, \ \ \h_{-1}=H_0^1\times H^{-1},\]
 which are  equipped with  usual graph norms. For example,
 \[\|(u,v)\|_{\mathcal{H}}^2=\|u\|_{H^1}^2+\|u\|_{p+1}^2+\|v\|^2.\]

\begin{assumption}\label{11}
\begin{description}
  \item (i) $f\in C^1(\r)$ and
\begin{equation}\label{1.3}
          c_0|s|^{p-1}-c_1\leq f'(s)\leq c_2(1+|s|^{p-1}), \ \ \forall  s\in \r,
        \end{equation}
  with some $\frac{N+2}{N-2}=p^*<p<p^{**}=\frac{N+4}{(N-4)^+}$, where $c_i$ are positive constants and $a^+=\max\{a,0\}$;
 \item (ii) $(u^\tau_0, u^\tau_1)\in\h$ with $\|(u^\tau_0, u^\tau_1)\|_\h\leq R, g, \partial_t g\in L^2_b(\r; L^2)$, where
\begin{equation*}
        L^2_b(\r; L^2)=\{\phi\in L^2_{loc}(\r; L^2) | \|\phi\|^2_{L^2_b(\r; L^2)}=\sup_{t\in\r}\int_t^{t+1}\|\phi(s)\|^2ds< +\infty\}.
      \end{equation*}
      \end{description}
\end{assumption}

When $N = 1$,  Eq. \eqref{1.1}, without strong damping  $-\Delta u_{t}$, was introduced by Kirchhoff \cite{Kirchhoff} to describe the nonlinear  vibrations of an elastic stretched string.   In real process, dissipation plays an important spreading role for
the energy gather arising from the nonlinearity.  So the researches on the  Kirchhoff  wave equations with different  type of  dissipations have attracted considerable attention, the
well-posedness and asymptotic behavior of solutions to the Kirchhoff wave models  with dissipation $-\Delta u_{t}$  or $u_t$ or $h(u_t) $ (with $h(s)s\geq 0$) have been well investigated by many authors
(see \cite{Bae, Cavalcanti,  Matsuyama,  Nakao1, Nishihara, Ono, Ono1}  and references therein).

 Recently, Chueshov \cite{Chueshov} studied the well-posedness and longtime dynamics for the autonomous   Kirchhoff wave model  with strong nonlinear damping
\begin{eqnarray}\label{1.4}
u_{tt}-\sigma(\|\nabla u\|^{2}) \Delta u_{t}-\phi(\|\nabla u\|^{2})\Delta u+f(u)=h(x).
\end{eqnarray}
A major breakthrough is  that he finds a supercritical
exponent $p^{**} \equiv \frac{N+4}
{(N-4)^+}$  and showes that when the growth exponent $p$ of the nonlinearity $f(u)$ is up to the
supercritical range: $1 \leq  p< p^{**}$,  the IBVP of Eq. \eqref{1.4} is still  well-posed  and the related solution semigroup has a  partially strong  global attractor $\mathcal{A}_{ps}$, i.e., the compactness and attractiveness of $\mathcal{A}_{ps}$ are in the phase space $\h_{ps}=(H^1_0\cap L^{p+1,w})\times L^2$, which is  equipped with the  partially strong topology:
\begin{align} &
  (u^n,v^n)\rightarrow (u,v)\  \hbox{in} \  \h_{ps} \  \hbox{if and only if}\nonumber\\
  &  (u^n,v^n)\rightarrow (u,v)\  \hbox{in} \  H^1_0 \times L^2 \  \hbox{and}\  u^n \rightharpoonup u\  \hbox{in} \  L^{p+1},\label{1.5}
\end{align}
where the sign $``\rightharpoonup"$ denotes weak convergence. In particular, in the non-supercritical
case: $1\leq p\leq p^*\equiv\frac{N+2}{N-2}$,  the partially strong topology becomes the  strong one. By the way, here the growth exponent $p^*$ is said to be  critical relative to the natural energy space $\h=(H_0^1\cap L^{p+1})\times L^2$ for $H^1\hookrightarrow L^{p+1}$ as $p\leq p^*$, but the Sobolev embedding ceases to be effective as $p>p^*$.
  For the related researches  on this topic, one can see also \cite{DYL, Kalantarov,M-Z}.   Recently,    Ding, Yang and Li \cite{DYL}   removed the restriction of partially strong topology in \cite{Chueshov}.

 Uniform attractor and pullback
attractor   (see Def. \ref{22'} and Def. \ref{22} below) are two basic concepts to study the longtime dynamics of   non-autonomous evolution equations with various dissipations (cf. \cite{C-V,MM,Sun1,  BX}).
 Although there have been some  researches  on the global attractors of autonomous Kirchhoff wave equations with strong damping (cf. \cite{Chueshov,Kalantarov,M-Z, Nakao1, Nakao2, Y1, Y-W, Y-D}), there are only a few recent results on the longtime dynamics of more complicated non-autonomous ones (\cite{zhou,zhong}). We refer to \cite{zhou} for the investigations  on the existence of the kernel $\mathcal{K}$ and   the
Hausdorff dimension of the kernel sections $\mathcal{K}(s)$ for  strongly damped non-autonomous Kirchhoff wave models
 \begin{align}
u_{tt}-\alpha\Delta u_{t}-(\beta+\gamma\|\nabla u\|^\frac{\rho}{2})\Delta u+h(u_t)+f(u,t)=g(x,t)
\end{align}
in a bounded domain $\Omega\subset \mathbb{R}^N (N=1,2,3)$ with Dirichlet boundary condition, where $\alpha>0, \beta>0, \rho\geq -1, \gamma\geq 0$ and the source term $f(u,t)$ is of subcritical growth on $u$. 

 Recently, Wang and Zhong \cite{zhong}  studied  the existence and the upper semi-continuity of pullback attractors of problem \eqref{1.1}-\eqref{1.2}.
 Under  the critical nonlinearity assumptions:
\begin{align}\label{1.6}
&
f'(u)\geq -l,\ \ |f'(u)|\leq C(1+|u|^{p^*-1}),\nonumber\\
&\liminf_{|u|\rightarrow\infty}\frac{uf(u)-kF(u)}{u^{2}}\geq 0,\ \ \liminf_{|u|\rightarrow\infty}\frac{F(u)}{u^{2}}\geq 0,
\end{align}
where $F(u)=\int_0^u f(s)ds$,  they  established the existence of   pullback attractors and their  upper semicontinuity on the perturbed parameter $\epsilon$.

But there are  still some unsolved questions. For example,  for the perturbed non-autonomous   Kirchhoff wave
model \eqref{1.1},  if the nonlinearity $f(u)$ is of the supercritical growth $p: p^*\leq  p<p^{**}$, what about the existence and structure of its  uniform attractor and pullback  attractor?    What about the stability  of the attractors on the  perturbed parameter $\epsilon$?

The purpose of the present paper is to solve  these  questions. It proves that in  supercritical nonlinearity case  $p^*< p<p^{**}$:

(i) the related family of processes $\{U^\epsilon_{g}(t,\tau)\},  g\in \Sigma$  has in $\h$ a  compact uniform attractor $\mathcal{A}^\epsilon_\Sigma$  for each $\epsilon\in [0,1]$ and its structure is shown  (see Theorem \ref{43});

(ii) the family of compact uniform attractor $\mathcal{A}^\epsilon_\Sigma$   is upper semicontinuous  on the   perturbed parameter $\e$  in the sense of $\h_{ps}$ topology (i.e.,  partially strong topology) (see Corollary \ref{54}).

As a consequence, for any fixed $g\in \Sigma$ (the symbol space), the family of all kernel sections $\mathcal{A}^\epsilon_{g}=\{\mathcal{K}^\epsilon_{g}(t)\}_{t\in \mathbb{R}}$ is  the  pullback attractor of the  process $\{U^\epsilon_{g}(t,\tau)\}$ in $\h$ for each $\epsilon\in [0,1]$ (cf. \cite{C-V}), and it is also upper semicontinuous  on   $\e$  in the sense of $\h_{ps}$ topology (see Corollary \ref{54}).

In particular, for autonomous case, i.e., $g(x,t)\equiv g(x)$,  the related process $\{U^\epsilon_{g}(t,\tau)\}$ becomes the solution semigroup $S^\epsilon(t)$ acting on the phase space $\h$ for each $\epsilon\in [0,1]$, and the related pullback attractor becomes the global attractor $\mathcal{A}^\e$ of $S^\epsilon(t)$ in $\h$, which is   upper semicontinuous  on   $\e$  in the sense of $\h_{ps}$ topology.

 The main contributions of the paper are that under the assumptions that the external force  $g$ is translation bounded (rather than translation compact as usual),
 and the  the nonlinearity $f(u)$ is of
 supercritical growth $p:  p^*< p<p^{**}$,
 by combining  newly developed criterion of compensated compactness  \cite{Sun},
quasi-stabilizability estimates method \cite{Chueshov1}  and  J. Ball's  technique \cite{Ball},  we   prove  the   existence of the uniform attractor $A^\epsilon_\Sigma$ of problem \eqref{1.1}-\eqref{1.2} and   show their upper semicontinuity on the perturbed parameter $\epsilon$ in the sense of partially strong topology.   These results not only extend Chueshov's
work  on  autonomous  Kirchhoff
 models in \cite{Chueshov} to non-autonomous ones but also extend  Wang and Zhong's results  on pullback attractor \cite{zhong} to the supercritical nonlinearity case.

Recently, many authors devote to study the uniform attractor of non-autonomous
dissipative PDEs with non  translation compact external forces.  They
 introduce several new classes of external forces that
are not translation compact, but nevertheless allow the attraction in a strong
topology of the phase space and give  some   criteria   on this kind of   uniform attractor and applications of them (cf. \cite{Lu1,Lu2, Lu3,Ma2, Ma3,Moise,Sun,Zelik2,Zelik}).

 We show in the present paper  that the weak solutions of non-autonomous
Kirchhoff wave  model \eqref{1.1}-\eqref{1.2}  are   of higher partial regularity when $t>\tau$, which results in that not only
 the requirement for the external force $g: g, \partial_t g\in L^2_b(\mathbb{R};L^2)$  is natural but also  permits  non  translation compact external forces $g$.

The paper is organized as follows. In Section 2,  we introduce some  preliminaries. In Section 3, we give some results on the well-posedness.  In Section 4, we discuss the existence of uniform attractors. In Section 5, we investigate  the upper semicontinuity of the uniform attractors on the perturbed parameter  $\e$.

\section{Preliminaries}

\begin{definition} (i) The family of sets $\{U_\sigma(t,\tau)| t\geq \tau, \tau\in \r\}, \sigma\in \Sigma$ (parameter set) is said to be  a family of processes acting on   Banach  space $E$   if for each $\sigma\in \Sigma$,  $\{U_\sigma(t,\tau)| t\geq \tau, \tau\in \r\}$
is a process acting on $E$, i.e., the two-parameter   mappings  from $E$ to $E$
satisfying
\begin{align*} &
U_\sigma(t,s)U_\sigma(s,\tau)=U_\sigma(t,\tau), \ \ \forall  t\geq s\geq \tau, \tau\in \r, \\
&
U_\sigma(\tau,\tau) = I \ \ \hbox{(identity operator)},\ \ \tau\in \r.
\end{align*}
And the set  $\Sigma$ is  said to be   the symbol space and $\sigma\in \Sigma$ to be a symbol.

(ii) Let $\{T(t)\}_{t\geq 0}$ be a translation semigroup acting on $\ls$. The family of processes $\{U_\s(t,\tau)\},\s\in\ls$ is said to be satisfy  the translation identity if
\begin{equation}\label{2.1}
  U_\s(t+s,\tau+s)=U_{T(s)\s}(t,\tau), \ \ \forall \s\in\ls, t\geq \tau,\tau\in\r, s\geq 0.
\end{equation}

(iii) A bounded subset $B_0\subset E$ is said to be a bounded uniformly ($w. r. t.\  \s\in\ls$) absorbing set of the family of processes $\{U_\s(t,\tau)\},\s\in\ls$ if for any $\tau\in\r$ and bounded subset $B\subset X$ there exists a $T_0=T_0(B, \tau)\geq \tau$ such that
\[\bigcup_{\s\in\ls}U_\s(t,\tau)B\subset B_0,\ \ \forall t\geq T_0.\]
\end{definition}

\begin{definition}\label{22'}
A family of nonempty compact subsets $\{\mathcal{A}(t)\}_{t\in\mathbb{R}}$ of $E$ is said to be  a pullback attractor of the process $U(t,\tau)$ if
it is  invariant, i.e.,
             $  U(t,s)\mathcal{A}(s) =\mathcal{A}(t),  \ t\geq s $, and
  it pullback attracts all the bounded subsets of $E$, i.e., for every bounded subset $D\subset E$ and $t\in\mathbb{R}$,
          \begin{equation*}
            \lim_{s\rightarrow +\infty}\mathrm{dist}_E\{U(t,t-s)D, \mathcal{A}(t)\}=0.
           \end{equation*} 
 Here, $\mathrm{dist}_E\{\cdot,\cdot\}$ is the Hausdorff semidistance in $E$, i.e.,
 \begin{equation*}
 \mathrm{dist}_E\{A, B\}=\sup_{x\in A}\inf_{y\in B}\|x-y\|_E,  \ \  A, B \subset E.
 \end{equation*}
\end{definition}

\begin{definition}\label{22}
A closed set $\mathcal{A}_\Sigma\subset E$ is said to be the uniform ($w.r.t.\ \sigma\in\Sigma$) attractor of the family of processes $\{U_\sigma(t,\tau)\},\sigma\in\Sigma$  if
\begin{description}
  \item (i) (Attractiveness)  $\mathcal{A}_\Sigma$  uniformly ($w.r.t.\ \sigma\in\Sigma$) attracts all the bounded  subsets in $E$, i.e., for every bounded subset $B\subset X$ and $\tau \in\r$,
         \begin{equation*}
           \lim_{t\rightarrow\infty}\sup_{\sigma\in\Sigma}\mathrm{dist}_E\{U_\sigma(t,\tau)B,\mathcal{A}_\Sigma\}=0;
         \end{equation*}
         \item (ii) (Minimality)  for  any closed set $\mathcal{A}'\subset E$, if $\mathcal{A}'$ is of property (i), then $\mathcal{A}_\Sigma\subset \mathcal{A}'$.
\end{description}
\end{definition}

\begin{definition}\label{23}
(i) For any fixed $\sigma\in \Sigma$,  the set of all  bounded
full trajectories of the process $U_\sigma(t,\tau)$:
\begin{equation*}
  \K_\sigma= \{u(\cdot)| U_\sigma(t,\tau) u(\tau)=u(t),  \|u(t)\|_E\leq C_u, \forall  t\geq \tau, \tau\in \r \}
\end{equation*}
is said to be the kernel of the process  $U_\sigma(t,\tau)$. The set $\K_\sigma(s)=\{u(s)| u(\cdot)\in \K_\sigma\}$ is said to be the kernel section at time $t=s, s\in \r$.

(ii) The family of processes $\{U_\sigma(t,\tau)\},\sigma\in\Sigma$ is said to be uniformly ($w.r.t.\ \sigma\in\Sigma$) asymptotically compact on $E$, if  for any   $\tau\in\r$, bounded sequences
$\{\xi_n\}\subset E, \{\s_n\}\subset \ls$ and   sequence  $\{t_n\} \subset \r$ with $t_n\geq \tau $ and $t_n\rightarrow +\infty$, the sequence $\{U_{\s_n}(t_n,\tau)\xi_n\}$ is precompact in $E$ (cf. \cite{Moise}).

(iii)  The family of processes $\{U_\sigma(t,\tau)\},\sigma\in\Sigma$ is said to be norm-to-weak continuous,  if for any fixed $t$ and $\tau\in\r$ with $t\geq \tau$, for any sequence $\{(x_n,\s_n)\}\subset E\times \ls, (x_n,\s_n) \rightarrow (x,\s) $ in $E\times \ls$  imply that $U_{\s_n}(t,\tau)x_n\rightharpoonup U_\s(t,\tau)x$ in $E$.
\end{definition}

\begin{lemma}\label{26}\cite{Sun} Assume that $\ls$ is a compact metric space, the translation semigroup $\{T(t)\}_{t\geq 0}$ is continuous in $\ls$, the family of processes $\{U_\s(t,\tau)\}, \s\in\ls$ satisfies the translation identity \eqref{2.1} and
\begin{description}
   \item(i) it  is norm-to-weak continuous;
  \item (ii) it  has a bounded uniformly ($w. r. t.\  \s\in\ls$) absorbing set $B_0$ in $E$;
  \item (iii)  it is uniformly ($w.r.t.\ \s\in\ls$) asymptotically compact in $E$.
\end{description}
Then  it  has a compact uniform ($w.r.t.\ \s\in\ls$) attractor $\mathcal{A}_\ls$, and
\begin{equation}\label{2.2}
  \mathcal{A}_\ls=\omega_{0,\ls}(B_0)=\bigcup_{\s\in\ls}\K_\s(s), \ \ \forall s\in\r,
\end{equation}
where $\K_\s$ is the kernel of the process $U_\s(t,\tau), \omega_{0,\ls}(B_0)$ is the uniform $\omega$-limit set of $B_0$ at $t=0$, i.e.,
\begin{equation}\label{2.3}
 \omega_{0,\ls}(B_0)=\bigcap_{t\geq 0}\Big[\bigcup_{\s\in\ls}\bigcup_{s\geq t}U_{\s}(s,0)B_0\Big]_E,
\end{equation}
and the sign $[ \, \cdot\,  ]_E$ denotes the closure in $E$.
\end{lemma}

\begin{definition}\label{27}  Let $\ls$ be a symbol space and  $B$ be a bounded subset in  Banach space $E$. A function $\phi(\cdot,\cdot;\cdot,\cdot)$  defined on $(B\times B)\times (\ls\times \ls)$ is said to be a  contractive function if for any sequences $\{x_n\} \subset B$ and $\{\s_n\} \subset \ls$, there exist    subsequences $\{x_{n_k}\} \subset\{x_n\} $ and $\{\s_{n_k}\} \subset\{\s_n\} $ such that
\begin{equation*}
  \lim_{k\rightarrow\infty}\lim_{l\rightarrow\infty}\phi(x_{n_k},x_{n_l};\s_{n_k},\s_{n_l})=0.
\end{equation*}
\end{definition}

\begin{lemma}\label{28}\cite{Sun}  Assume that  the family of processes $\{U_\sigma(t,\tau)\},\sigma\in\Sigma$ satisfies   translation identity \eqref{2.1}, and the following conditions holds:
\begin{description}
   \item(i)
 it has a  bounded uniformly ($w. r. t.\  \s\in\ls$) absorbing set $B_0\subset E$;

  \item(ii)  for any $\delta>0$ there exist $T=T(B_0, \delta)>0$ and a contractive function $\phi_T$ defined on $(B_0\times B_0)\times (\ls\times\ls)$ such that
\begin{equation*}
  \|U_{\s_1}(T,0)x-U_{\s_2}(T,0)y\|_E\leq \delta+\phi_T(x,y;\s_1,\s_2), \ \ \forall x,y \in B_0, \ \s_1,\s_2 \in\ls.
\end{equation*}
\end{description}
Then the family of processes $\{U_\sigma(t,\tau)\},\sigma\in\Sigma$ is uniformly ($w.r.t.\ \sigma\in\Sigma$) asymptotically compact on $E$.
\end{lemma}

\begin{lemma}\label{29}\cite{Lu1}  Let the family of processes $\{U_\sigma(t,\tau)\},\sigma\in\Sigma$ satisfy the translation identity \eqref{2.1} and the symbol space $\ls$ be translation invariant, i.e., $T(h)\ls=\ls$ for all $h\geq 0$. Then for every $\tau\in\r$ and $\s\in\ls$, there exists at least one $\s'\in\ls$ satisfying
\begin{equation*}
  U_\s(t,\tau)=U_{\s'}(t-\tau+\tau_0, \tau_0), \ \ \forall t\geq \tau, \tau_0\in\r.
\end{equation*}
\end{lemma}

\begin{lemma} \label{210}\cite{Simon}   Let $X, B$ and $Y$
be  Banach spaces, $X \hookrightarrow\hookrightarrow B\hookrightarrow
Y,$
\begin{eqnarray}&&
W =\{u\in L^p(0,T;X)|  u_t\in L^1(0,T;Y)\}, \ \hbox{with}\
1\leq p<\infty,\nonumber\\
&& W_1=\{u\in L^\infty(0,T;X)| u_t\in L^r(0,T;Y)\},\ \hbox{with}\
r>1.\nonumber
\end{eqnarray}
 Then,
 \[ W \hookrightarrow\hookrightarrow
L^p(0,T;B),\ \ W_1\hookrightarrow\hookrightarrow C([0,T];B).\]
\end{lemma}

\section{Well-posedness}
In this section, we  discuss the well-posedness of  problem \eqref{1.1}-\eqref{1.2}. We first define a  symbol space  generated by a fixed external force term $g_0$, with  $g_0, \partial_t g_0\in L^2_b(\r; L^2)$.

 Define the translation operator
\begin{equation*}
 T(h):L^2_{loc}(\r;L^2)\rightarrow L^2_{loc}(\r;L^2),\ \
  T(h)g(s)=g(s+h),\ \   s,  h\in \r.
\end{equation*}
Obviously, $\{T(h)\}_{ h\in \r}$ constitutes  a translation  group on $L^2_{loc}(\r;L^2)$. Let
\begin{equation}\label{3.1}
 \ls_0=\{T(h)g_0| h\in\r\}, \ \  \ls=\mathcal{H}(g_0)=\big[\ls_0\big]_{L^{2,w}_{loc}(\r; L^2)},
\end{equation}
and   $\ls$  be  equipped with   $L^{2,w}_{loc}(\r; L^2)$ topology, i.e.,
\[u^n\rightarrow u\ \ \hbox{in}\ \ \ls\ \ \hbox{if and only if}\ \  u^n\rightharpoonup u\ \ \hbox{in}\ \  L^{2}(t_1, t_2; L^2), \ \  \forall\  [t_1, t_2]\subset \r.\]
Then   $\ls$ is a compact metric space,
\begin{equation}\label{3.2}
\sup_{g\in\ls}\|g\|_{L^2_b(\r;L^2)}\leq \|g_0\|_{L^2_b(\r;L^2)}, \ \ \sup_{g\in\ls}\|\partial_tg\|_{L^2_b(\r;L^2)}\leq \|\partial_tg_0\|_{L^2_b(\r;L^2)},
\end{equation}
and $\{T(t)\}_{t\in\r}$ is continuous and invariant in $\ls$, i.e., $T(h)\ls=\ls,\ \forall h\in\r$ (cf. \cite{C-V}).

Repeating  the same arguments as  in \cite{Chueshov} (where the well-posedness of   problem \eqref{1.1}-\eqref{1.2} has been established  for the autonomous case:  $g(x,t)\equiv g(x)$ ) except for the treatment of $g(x,t)$  one easily  gets the following theorem.

\begin{theorem}\label{31} Let Assumption  \ref{11} be valid, with $g\in\ls$. Then   problem \eqref{1.1}-\eqref{1.2} admits a unique weak solution $u^\epsilon$, with $(u^\epsilon,u^\epsilon_t)\in C([\tau, T]; \h)$ for each $\epsilon\in [0,1]$, and
\begin{equation}\label{3.3}
  \|(u^\epsilon,u^\epsilon_t)(t)\|^2_\h+\|u^\epsilon_{tt}\|^2_{H^{-2}}+\int_\tau^T\|\nabla u_t^\epsilon(s)\|^2ds\leq K, \ \  t\in [\tau, T],
\end{equation}
where $K=C(\tau, T, R, \|g\|_{L^2_b(\r; L^2)})$ is a positive constant. Moreover, the solution is of the following properties:
\begin{description}
  \item(i) (Partial regularity when $t>\tau$)
\begin{equation}\label{5.19}
  \|\nabla u^\epsilon_t(t)\|^2+\|u^\epsilon_{tt}(t)\|^2_{H^{-1}}\leq K_1\Big(1+\frac{1}{(t-\tau)^2}\Big), \ \ t\in(\tau, T],
\end{equation}
where  $K_1=C(T-\tau, R,\|g_0\|_{L^2_b(\r; L^2)}, \|\partial_tg_0\|_{L^2_b(\r; L^2)})$;
  \item(ii) (Energy identity)
  \begin{equation}\label{3.4}
    E(u^\epsilon(t),u^\epsilon_t(t))+\int_s^t\big[\|\nabla u^\epsilon_t(r)\|^2-(g, u^\epsilon_t)\big]d r=E(u^\epsilon(s),u^\epsilon_t(s)), \ \ \forall t>s\geq \tau,
  \end{equation}
  where
  \begin{equation*}
  E(u^\epsilon,u^\epsilon_t)=\frac{1}{2}\Big[\|u^\epsilon_t\|^2+\|\nabla u^\epsilon\|^2+\frac{\e}{2}\|\nabla u^\epsilon\|^4 +2(F(u^\epsilon),1)\Big]\ \ \hbox{with} \ \ F(s)=\int_0^s f(r) d r;
  \end{equation*}
  \item (iii) (Stability and quasi-stability in $\h_{-1}$) the following Lipschitz stability
  \begin{equation}\label{3.5}
  \|(z,z_t)(t)\|^2_{\h_{-1}}\leq K\Big[\|(z,z_t)(\tau)\|^2_{\h_{-1}}+\|g_1-g_2\|^2_{L^2(\tau,t;H^{-1})} \Big],\ \ t\in[\tau, T],
  \end{equation}
  and quasi-stability
   \begin{equation}\label{3.6}
   \begin{split}
    \|(z,z_t)(t)\|^2_{\h_{-1}}  & \leq e^{-\kappa(t-\tau)}\|(z,z_t)(\tau)\|^2_{\h_{-1}} \\
       & +K \int_\tau^t\Big[\|(z,z_t)(s)\|^2_{L^2\times H^{-2}}+\|(g_1-g_2)(s)\|^2_{H^{-1}}\Big]ds, \ \ t\in[\tau,T],
   \end{split}
  \end{equation}
 hold for $z=u^{\epsilon, 1}-u^{\epsilon, 2}$, where $u^{\epsilon, 1}, u^{\epsilon, 2}$ are two weak solutions of problem \eqref{1.1}-\eqref{1.2}
 corresponding to initial data $(u^{\epsilon, i}(\tau),u^{\epsilon, i}_t(\tau))\in \h$,  with $\|(u^{\epsilon, i}(\tau),u^{\epsilon,  i}_t(\tau))\|_\h\leq R$, and $g_i\in L^2_b(\r; L^2)$, respectively.
 \end{description}
\end{theorem}
\medskip

 For any $g\in\ls$,  we  define the solution operator
\begin{equation*}\label{3.7}
 U_g^\e(t,\tau):\h\rightarrow \h,\ \  U^\e_g(t,\tau)(u^\tau_0, u^\tau_1)=(u^\e,u^\e_t)(t), \ \ t\geq \tau,
\end{equation*}
where $u^\e$ is a weak solution of problem \eqref{1.1}-\eqref{1.2}.  Theorem \ref{31} shows that $\{U^\e_g(t,\tau)\},g\in\ls, \e\in[0,1]$ is a family of processes acting on the phase space $\h$. The uniqueness of weak solutions implies the translation identity
\begin{equation}\label{3.8}
U^\e_g(t+s,\tau+s)=U^\e_{T(s)g}(t,\tau), \ \ \forall t\geq \tau, \ \tau\in\r, \  s\geq 0, \  \e\in[0,1].
\end{equation}

\section{Existence of uniform attractors}

 For simplicity,    we omit the superscript $\epsilon$ and denote $u=u^\epsilon$ in the following.

\begin{lemma}\label{41} Let Assumption  \ref{11} be valid, with $g\in\ls$. Then
\begin{enumerate}[$(i)$]
  \item  For any  sequence $\{(\xi_n, g_n)\}\subset \h\times \ls$ with $(\xi_n,g_n)\rightarrow (\xi,g)$ in $\h_{-1}\times \ls$, we have
\begin{equation}\label{4.1}
  U^\e_{g_n}(t,\tau)\xi_n\rightarrow U^\e_g(t,\tau)\xi\ \  \hbox{in} \ \ \h_{-1}, \ \ \forall \e\in[0,1].
\end{equation}
  \item The family of processes $\{U^\e_g(t,\tau)\},g\in\ls$ is norm-to-weak continuous for each $\e\in [0,1]$.
\end{enumerate}
\end{lemma}
\begin{proof} (i)  The fact $g_n\rightarrow g$ in $\ls$ implies that
\begin{equation}\label{4.2}
  g_n\rightarrow g\ \ \hbox{in}\ \ L^2(\tau,t;H^{-1}),\ \ \forall t>\tau.
\end{equation}
Indeed, it follows from estimate \eqref{3.2} that both the  sequences  $\{g_n\} $ and $\{\partial_tg_n\} $ are  bounded in $L^2(\tau,t;L^2)$, which implies  that  $\{g_n\} $ is precompact in $L^2(\tau,t;H^{-1})$ for $L^2\hookrightarrow\hookrightarrow H^{-1}$ (see Lemma \ref{210}). So formula \eqref{4.2} holds. The combination of \eqref{4.2}  and stability estimate \eqref{3.5} yields \eqref{4.1}.

(ii)   Let $(\xi_n,g_n)\rightarrow (\xi,g)$ in $\h\times \ls$. By \eqref{4.1},
\begin{equation*}
 U^\e_{g_n}(t,\tau)\xi_n= (u^{ n},u_t^{ n})(t)\rightarrow (u,u_t)(t) \ \  \hbox{in} \ \ \h_{-1}.
\end{equation*}
 By the boundedness of $\{(u^{ n},u_t^{ n})(t)\}$ in $\h$ (see \eqref{3.3}),
\begin{equation}\label{4.4}
 u^{ n}(t)\rightharpoonup u(t)\ \ \hbox{in}\ \  L^{p+1},\ \ u_t^{ n}(t)\rightharpoonup u_t(t)\ \ \hbox{in} \ \ L^2.
\end{equation}
Therefore,
\begin{equation*}
 (u^{ n},u_t^{ n})(t)\rightharpoonup (u,u_t)(t) \ \  \hbox{in} \ \ \h.
\end{equation*}
\end{proof}

\begin{lemma}\label{42}
Let Assumption  \ref{11} be valid, with $g\in\ls$. Then the family of processes $\{U^\e_g(t,\tau)\}, g\in\ls, \e\in[0,1]$ has a uniformly ($w. r. t.\  g\in\ls$ and $\e\in[0,1]$) absorbing set $\b=\{\xi\in\h|\|\xi\|_\h\leq R_0\}$.
\end{lemma}
\begin{proof}  Using the multiplier $u_t+\delta u \ (= u^\e_t+\delta u^\e$) in Eq. \eqref{1.1}, we obtain
\begin{equation*}
  \frac{\d}{\d t}\Gamma(\xi_u(t))+\Psi(\xi_u(t))=0,
\end{equation*}
where $\xi_u=(u,u_t)$,
\begin{align}
       & \Gamma(\xi_u)=\frac{1}{2}\Big[\|u_t\|^2+\|\nabla u\|^2+\frac{\e}{2}\|\nabla u\|^4+2(F(u),1)\Big]+\delta\Big[\frac{1}{2}\|\nabla u\|^2+(u_t,u)\Big],\label{4.0} \\
      & \Psi(\xi_u)=\|\nabla u_t\|^2-\delta\|u_t\|^2+\delta\Big[\|\nabla u\|^2+\e\|\nabla u\|^4+(F(u),u)\Big]-(g,u_t+\delta u).\nonumber
 \end{align}
Assumption  \eqref{1.3} implies that
\begin{equation}\label{4.4}
  \begin{split}
  &\frac{c_0}{2p}|u|^{p+1}-C\leq f(u)u\leq C(1+|u|^{p+1}),\\
&\frac{c_0}{2p(p+1)}|u|^{p+1}-C\leq F(u)\leq C(1+|u|^{p+1}),\\
  & f(u)u-F(u)+\frac{c_1}{2}|u|^2\geq 0.
  \end{split}
\end{equation}
Thus  a simple calculation shows that
\begin{equation*}
  \begin{split}&
  \kappa \|\xi_u\|_\h^2-C\leq \frac{1}{4}\|u_t\|^2+\frac{1}{2}\|\nabla u\|^2+\frac{c_0}{2p(p+1)}\|u\|_{p+1}^{p+1}-C
      \\
     &\leq \Gamma(\xi_u)\leq C\Big[\|u_t\|^2+\|\nabla u\|^2+\|u\|_{p+1}^{p+1}+\e\|\nabla u\|^4+1\Big],  \\
            &
          \Psi(\xi_u)  \geq \Big(1-\frac{2\delta}{\lambda_1}\Big)\|\nabla u_t\|^2+\delta\Big[\frac{1}{2}\|\nabla u\|^2+\e\|\nabla u\|^4+\frac{c_0}{2p}\|u\|_{p+1}^{p+1}\Big]-C(1+\|g\|^2) \\
            & \ \ \ \ \ \ \ \ \ \ \geq \kappa \Gamma(\xi_u)-C(1+\|g\|^2)
         \end{split}
\end{equation*}
for $\delta>0$ suitably small, where $\lambda_1$ is the first eigenvalue of $-\Delta$ with Dirichlet boundary condition and $\kappa$ is a small positive constant. Hence,
\begin{align} &
  \frac{\d}{\d t}\Gamma(\xi_u(t))+\kappa \Gamma(\xi_u(t))\leq C(1+\|g(t)\|^2),\nonumber\\
& \|\xi_u(t)\|_\h^2\leq Q(\|\xi_u(\tau)\|_\h)e^{-\kappa(t-\tau)}+C(1+\|g_0\|^2_{L^2_b(\r;L^2)}),\ \ \forall t\geq \tau,\label{4.5}
\end{align}
for all $\e\in [0,1], g\in \ls$ and $\xi_u(\tau)\in\h$, where $\xi_u(t)=U^\e_g(t,\tau)\xi_u(\tau), Q$ is a monotone positive function.

Let
\begin{equation*}
  \b=\{\xi\in\h|\|\xi\|_\h\leq R_0\}\ \ \hbox{with}\ \ R_0^2=2C(1+\|g_0\|^2_{L^2_b(\r;L^2)}).
\end{equation*}
Estimate \eqref{4.5} shows that  $\b$ is a uniformly ($w. r. t.\  g\in\ls$ and $\e\in [0,1]$) absorbing set of the family of processes $\{U^\e_g(t,\tau)\},g\in\ls, \e\in[0,1]$.
\end{proof}

\begin{theorem}\label{43} Let Assumption  \ref{11} be valid, with $g\in\ls$. Then  the family of processes $\{U^\e_g(t,\tau)\}, g\in\ls$ has in $\h$ a compact uniform ($w. r. t.\  g\in\ls$) attractor $\mathcal{A}_\ls^\e$  for each $\e\in [0,1]$, and
\begin{equation}\label{4.6}
\mathcal{A}_\ls^\e=\omega^\e_{0,\ls}(\b)=\bigcup_{g\in\ls}\K^\e_g(s), \ \ \forall s\in\r.
\end{equation}
\end{theorem}
\begin{proof}  Since the family of  processes  $\{U^\e_g(t,\tau)\}, g\in\ls$ satisfies  translation identity \eqref{3.8}, it is norm-to-weak continuous for each $\e \in [0,1]$ (see Lemma \ref{41}) and  has a uniformly
($w. r. t.\  g\in\ls$ and $\e \in [0,1]$) absorbing set $\b$ (see Lemma \ref{42}), by Lemma \ref{26},     it is sufficient to prove Theorem \ref{43} to show the precompactness of the sequence $\{U^\e_{g_n}(t_n,\tau)\xi_n\}$
 in $\h$, where $t_n\rightarrow+\infty$ as $n\rightarrow\infty$ (see Def. \ref{23}: (ii)). By translation identity \eqref{3.8},
\begin{equation}
U^\e_{g_n}(t_n,\tau)\xi_n=U^\e_{T(t_n)g_n}(0,\tau-t_n)\xi_n, \ \ \hbox{where}\ \ T(t_n)g_n\in \ls\ \ \hbox{and}\ \ \tau_n=\tau-t_n\rightarrow -\infty.
\end{equation}
Without loss of generality, it is enough to show that for every $\e\in [0,1]$,   any sequences $\{g_n\}\subset \ls$, $\{\xi_n\}\subset \b$ and $\tau_n\rightarrow -\infty$, the  sequence $\{U^\e_{g_n}(0,\tau_n)\xi_n\}$ is precompact in $\h$.

 Let
\begin{equation}\label{4.7}
 (u^n,u_t^n)(t)=U^\e_{g_n}(t,\tau_n)\xi_n, \ \ t\geq \tau_n.
\end{equation}
Due to Lemma \ref{29} (taking $\tau_0=0$ there) and the fact that $\b$ is a uniformly ($w. r. t.\  g\in\ls$) absorbing set of the family of processes $\{U^\e_g(t,\tau)\},g\in\ls$, there exists a positive constant $T_0$ independent of $\tau$  such that
\begin{equation}\label{4.8}
  \bigcup_{g\in\ls}U^\e_g(t,\tau)\b\subset \b, \ \ t\geq \tau+T_0, \ \  \forall \tau\in\r.
\end{equation}
For any fixed $T\in\n$, there exists a $N>0$ such that $-T\geq \tau_n+T_0$ as  $n\geq N$. Hence when $n\geq N$, by  \eqref{4.8},
\begin{equation*}
  U_{g_n}(t,\tau_n)\xi_n=U_{g_n}(t,-T)U_{g_n}(-T,\tau_n)\xi_n\in U_{g_n}(t,-T)\b,\ \ t\in[-T,0].
\end{equation*}
Therefore (see \eqref{3.3}),
\begin{align*}&
 \{u^n\} \ \ \hbox{is   bounded   in}\ \ L^\infty(-T,0;H_0^1\cap L^{p+1});\\
 &  \{u_t^n\} \ \ \hbox{is   bounded   in}\ \ L^\infty(-T,0; L^{2})\cap L^2(-T,0; H_0^1);\\
&
  \{u_{tt}^n\} \ \ \hbox{is  bounded  in}\ \ L^\infty(-T,0; H^{-2}),
\end{align*}
and (subsequence if necessary)
\begin{equation}\label{4.9}
  \begin{split}
     & (u^n,u_t^n)\rightarrow  (u,u_t)\ \ \hbox{weakly}^*\ \hbox{in}\ \ L^\infty(-T,0;\h);\\
     &(u^n,u_t^n)(t)\rightharpoonup (u,u_t)(t) \ \ \hbox{in}\ \ \h,\  t\in[-T,0];\\
&u_t^n\rightharpoonup u_t \ \ \hbox{in}\ \ L^2(-T,0; H_0^1); \\
&   g_n\rightarrow g\ \ \hbox{in}\ \ \ls,
  \end{split}
\end{equation}
where we have used  the compactness of $\ls$. By Lemma \ref{210},
\begin{align}\label{4.9'}
 &u^n\rightarrow  u \ \ \hbox{in}\ \  C([-T,0];H^{1-\delta})\ \ \hbox{and }\ \  a.e. \ (x,t)\in \Omega\times [-T,0];\nonumber\\
&
 u_t^n\rightarrow u_t \ \ \hbox{in}\ \ L^2(-T,0; L^2),
\end{align}
where $\delta\in(0,1)$.  It follows from  estimate \eqref{3.6} that
\begin{equation}\label{4.10}
\begin{split}
  \|U^\e_{g_1}(t,0)x-U^\e_{g_2}(t,0)y\|_{\h_{-1}}^2  \leq &e^{-\kappa t}\|x-y\|^2_{\h_{-1}}+C\|g_1-g_2\|^2_{L^2(0,t;H^{-1})} \\
    &  +C\int_0^t\|U^\e_{g_1}(s,0)x-U^\e_{g_2}(s,0)y\|^2_{L^2\times H^{-2}}ds,\ \ t\geq 0,
\end{split}
\end{equation}
for any  $\e\in [0,1],  x,y\in\b$ and $g_1, g_2\in\ls$, where  $C=C(t, R_0,  \|g_0\|_{L^2_b(\r;L^2)})$. For any sequence $\{g_n\}\subset \ls, \{g_n\}$ is  precompact  in $L^2(0,t;  H^{-1})$ for $g_n\rightarrow g$ in $\ls$ (see \eqref{4.2}).
By the similar arguments as \eqref{4.9'}, we obtain that
\begin{align*}
    \bigcup_{g\in\ls}U_g^\e(\cdot,0)\b \ \ \hbox{is   precompact  in}\ \ L^2(0,t; L^2\times H^{-2}).
    \end{align*}
Thus, it follows from \eqref{4.10} that for any $\delta> 0$, there exist $T=T(\b,\delta)>0$ and a contractive function
\begin{equation*}
  \Psi_T(x,y;g_1,g_2)=C\Big(\int_0^T\|g_1(s)-g_2(s)\|^2_{H^{-1}}+ \|U^\e_{g_1}(s,0)x-U^\e_{g_2}(s,0)y\|^2_{L^2\times H^{-2}}ds\Big)^{\frac{1}{2}}
\end{equation*}
defined on $(\b\times\b)\times(\ls\times\ls)$ such that
\begin{equation*}
 \|U^\e_{g_1}(t,0)x-U^\e_{g_2}(t,0)y\|_{\h_{-1}}\leq \delta+ \Psi_T(x,y;g_1,g_2).
\end{equation*}
By Lemma \ref{28}, the family of processes $\{U^\e_g(t,\tau)\}, g\in\ls$ is uniformly ($w. r. t.\  g\in\ls$) asymptotically compact in $\h_{-1}$. Therefore (subsequence if necessary),
\begin{equation}\label{4.11}
 (u^n,u_t^n)(-T)=U^\e_{g_n}(-T,\tau_n)\xi_n\rightarrow  (u,u_t)(-T) \ \ \hbox{in}\ \ \h_{-1}.
\end{equation}
By formula \eqref{4.1} and the uniqueness of the limit,
\begin{align}
 & (u^n,u_t^n)(t)=U^\e_{g_n}(t,-T)(u^n,u_t^n)(-T)\nonumber\\
 & \rightarrow U^\e_{g}(t,-T)(u,u_t)(-T)=(u,u_t)(t)\ \ \hbox{in}\ \  \h_{-1},\ \ \forall t\in[-T, 0]. \label{4.12}
\end{align}
So
\begin{equation}\label{4.13}
 u^n(t)\rightarrow u(t)\ \ \hbox{in}\ \ H^1_0,\ \  t\in [-T,0].
\end{equation}

By the standard diagonal process, we can extract a subsequence (still denoted by itself) such that \eqref{4.9} and \eqref{4.11}-\eqref{4.13} hold for all $T\in\n$.
\medskip

Rewrite energy identity \eqref{3.4} as the form
\begin{equation}\label{4.16'}
  \frac{\d}{\d t}E(u,u_t)+\|\nabla u_t\|^2=(g,u_t).
\end{equation}
Using the multiplier $\delta u$ in Eq. \eqref{1.1} and adding the resulting expression to \eqref{4.16'},
we obtain
\begin{equation}\label{4.15}
  \frac{\d}{\d t}\Gamma(u,u_t)+\delta \Gamma(u,u_t)=\Lambda(u,u_t),
\end{equation}
where $\Gamma(u,u_t)$ is as shown in \eqref{4.0} and
\begin{equation*}
  \begin{split}
    \Lambda(u,u_t) = &\delta\Big[\frac{3}{2}\|u_t\|^2-\frac{1}{2}\|\nabla u\|^2-\frac{3\e}{4}\|\nabla u\|^4+\delta\Big(\frac{1}{2}\|\nabla u\|^2+(u_t,u)\Big)\Big] \\
      &+(g, u_t+\delta u)-\|\nabla u_t\|^2-\delta\int_\lo[f(u)u-F(u)]dx.
  \end{split}
\end{equation*}
It follows from \eqref{4.15} that
\begin{equation}\label{4.16}
 \Gamma\Big((u^n,u^n_t)(0)\Big)=e^{-\delta T} \Gamma\Big((u^n,u^n_t)(-T)\Big)+e^{-\delta T}\int_{-T}^0e^{\delta s}\Lambda\Big((u^n,u^n_t)(s)\Big)ds,
\end{equation}
and the formula \eqref{4.16} also holds for $(u,u_t)$. By virtue of \eqref{4.9}-\eqref{4.9'}, \eqref{4.13} and the Lebesgue dominated convergence theorem,
\begin{align}
 & \lim_{n\rightarrow \infty} \int_{-T}^{0} e^{\delta s}\Big\{\delta\Big[\frac{3}{2}\|u^n_t\|^2-\frac{1}{2}\|\nabla u^n\|^2-\frac{3\e}{4}\|\nabla u^n\|^4+\delta\Big(\frac{1}{2}\|\nabla u^n\|^2+(u_t^n,u^n)\Big)\Big]+(g_n, u_t^n+\delta u^n)\Big\}ds\nonumber\\
      &= \int_{-T}^{0} e^{\delta s}\Big\{\delta\Big[\frac{3}{2}\|u_t\|^2-\frac{1}{2}\|\nabla u\|^2-\frac{3\e}{4}\|\nabla u\|^4+\delta\Big(\frac{1}{2}\|\nabla u\|^2+(u_t,u)\Big)\Big]+(g, u_t+\delta u)\Big\}ds.\label{4.17}
 \end{align}
It follows from \eqref{4.9} that
\begin{equation}\label{4.18}
 \int_{-T}^0e^{\delta s}\|\nabla u_t(s)\|^2ds\leq  \liminf_{n\rightarrow \infty}\int_{-T}^0e^{\delta s}\|\nabla u^n_t(s)\|^2ds.
\end{equation}
By \eqref{4.9'},
\begin{equation*}
f(u^n)u^n-F(u^n)+\frac{c_1}{2}|u^n|^2\rightarrow f(u)u-F(u)+\frac{c_1}{2}|u|^2\ \ a.e. \ (x,t)\in \Omega\times [-T,0].
\end{equation*}
hence by formula \eqref{4.4} and the Fatou lemma,
\begin{equation}\label{4.19}
  \int_{-T}^0\int_\lo e^{\delta s}[f(u)u-F(u)]dxds\leq \liminf_{n\rightarrow \infty}\int_{-T}^0\int_\lo e^{\delta s}[f(u^n)u^n-F(u^n)]dxds.
\end{equation}
The combination of \eqref{4.17}-\eqref{4.19} yields
\begin{equation}\label{4.20}
\limsup_{n\rightarrow \infty} \int_{-T}^0e^{\delta s}\Lambda\Big((u^n,u^n_t)(s)\Big)ds\leq \int_{-T}^0e^{\delta s}\Lambda\Big((u,u_t)(s)\Big)ds.
\end{equation}
Therefore, taking account of the boundedness of $\Gamma((u^n,u^n_t)(-T))$, we infer from \eqref{4.16} and \eqref{4.20} that
\begin{equation*}
\begin{split}
\limsup_{n\rightarrow \infty}\Gamma((u^n,u^n_t)(0))&\leq Ce^{-\delta T}+e^{-\delta T}\int_{-T}^0e^{\delta s}\Lambda((u,u_t)(s))ds\\
&=Ce^{-\delta T}+\Gamma((u,u_t)(0))-e^{-\delta T}\Gamma((u,u_t)(-T)).
\end{split}
\end{equation*}
Letting $T\rightarrow +\infty$, we obtain
\begin{equation*}
 \limsup_{n\rightarrow \infty}\Gamma((u^n,u^n_t)(0))\leq \Gamma((u,u_t)(0))\leq \liminf_{n\rightarrow \infty}\Gamma((u^n,u^n_t)(0)),
\end{equation*}
where we have used \eqref{4.8}-\eqref{4.9}, \eqref{4.12} and the Fatou lemma in the second inequality. Therefore,
\begin{equation*}
 \lim_{n\rightarrow \infty}\Gamma((u^n,u^n_t)(0))= \Gamma((u,u_t)(0)),
\end{equation*}
which implies (see \eqref{4.0})
\begin{equation}\label{4.21}
  \lim_{n\rightarrow \infty}\|u^n_t(0)\|=\|u_t(0)\|, \ \ \lim_{n\rightarrow \infty}\int_\lo F(u^n(0)) dx=\int_\lo F(u(0))dx.
\end{equation}
By \eqref{4.4}, the Fatou lemma and  \eqref{4.21},
\begin{equation*}
   \int_\Omega\Big( F(u(0))\pm C_1 |u(0)|^{p+1}\Big) dx\leq  \int_\Omega F(u(0))dx+\liminf_{n\rightarrow\infty}\pm C_1\int_\Omega|u^n(0)|^{p+1}dx,
 \end{equation*}
where $C_1= \frac{c_0}{2p(p+1)}$, that is,
\begin{align}
  &\limsup_{n\rightarrow\infty}\|u^n(0)\|^{p+1}_{p+1}\leq \|u(0)\|^{p+1}_{p+1}\leq \liminf_{n\rightarrow\infty}\|u^n(0)\|^{p+1}_{p+1},\nonumber\\
   &\|u(0)\|_{p+1}=\lim_{n\rightarrow \infty}\|u^n(0)\|_{p+1}.\label{4.22}
\end{align}
By \eqref{4.9},
\begin{equation}\label{4.14}
 u^n(0)\rightharpoonup u(0)\ \ \hbox{in}\ \ L^{p+1},\ \ u_t^n(0)\rightharpoonup u_t(0)\ \ \hbox{in}\ \ L^{2}.
\end{equation}
The combination of  \eqref{4.21}-\eqref{4.14} and  the uniform convexity of $L^{p+1}$ and $L^2$ yields
\begin{equation*}
  U^\e_{g_n}(0,\tau_n)\xi_n= (u^n, u^n_t )(0)\rightarrow (u, u_t)(0)\ \ \hbox{in}\ \ \h,
\end{equation*}
i.e., the family of processes $\{U^\e_g(t,\tau)\}, g\in\ls$ is uniformly ($w. r. t.\  g\in\ls$) asymptotically compact in $\h$. Therefore, by Lemma \ref{26}, we get the conclusion of Theorem \ref{43}.
\end{proof}

\section{Upper semicontinuity of the uniform attractors }
In this section, we discuss  the upper semicontinuity ($w. r. t.\ \e$) of the uniform attractors $\mathcal{A}_\ls^\e$.

\begin{theorem}\label{51} Let Assumption  \ref{11} be valid, with $g\in\ls$. Then the uniform attractors $\mathcal{A}_\ls^\e$ as shown in Theorem \ref{43} is upper semicontinuous at the point $\e_0\geq 0$ in the sense of $\h_{-1}$ topology, i.e.,
\begin{equation}\label{5.1}
 \lim_{\e\rightarrow \e_0}\mathrm{dist}_{\h_{-1}}\{\mathcal{A}_\ls^\e,\mathcal{A}_\ls^{\e_0}\}=0,
\end{equation}
and so does the kernel section $\K_g^\e(s)$, i.e.,
\begin{equation}\label{5.2}
 \lim_{\e\rightarrow \e_0}\mathrm{dist}_{\h_{-1}}\{\K_g^\e(s), \K_g^{\e_0}(s)\}=0,\ \ \forall g\in\ls, \ \ s\in \r.
\end{equation}
\end{theorem}
In order to  prove  Theorem \ref{51}, we first give   following lemmas.

\begin{lemma} (Lipschitz stability)\label{52}  Under the assumptions of Theorem \ref{51},  we have
\begin{equation}\label{5.3}
  \sup_{g\in\ls}\|U^{\e_1}_g(t,\tau)\xi_1-U^{\e_2}_g(t,\tau)\xi_2\|^2_{\h_{-1}}\leq C_K\Big(\|\xi_1-\xi_2\|^2_{\h_{-1}}+|\e_1-\e_2|^2\Big),\ \ t\in[\tau,T],
\end{equation}
   for any $\xi_i\in\b$ and $\e_i\in [0,1]\  (i=1,2)$, where $C_K=C(\tau, T, R_0, \|g_0\|_{L^2_b(\r; L^2)})$.
\end{lemma}
\begin{proof}Let $(u^i(t),u_t^i(t))=U^{\e_i}_g(t,\tau)\xi_i$.  It follows from estimate \eqref{3.3} that
\begin{equation}\label{5.4}
  \|(u^i,u_t^i)(t)\|_\h^2+\|u^i_{tt}(t)\|^2_{H^{-2}}+\int^T_\tau\|\nabla u^i_t(s)\|^2ds\leq K, \ \ \forall t\in[\tau, T],
\end{equation}
where $K=K(\tau, T, R_0, \|g_0\|_{L^2_b(\r; L^2)})>0$. Then $z=u^1-u^2$ solves
\begin{equation}
\begin{split}
   & z_{tt}-\Delta z_t-\Delta z+f(u^1)-f(u^2)=\e_1\|\nabla u^1\|^2\Delta u^1-\e_2\|\nabla u^2\|^2\Delta u^2,\ \ t>\tau, \label{5.5}\\
    & (z, z_t)(\tau)=\xi_1-\xi_2.
\end{split}
\end{equation}
Using the multiplier   $(-\Delta)^{-1}z_t+\delta z$ in  Eq. \eqref{5.5}, we obtain
\begin{equation}\label{5.6}
  \frac{\d}{\d t}\Phi(\xi_z)+(1-\delta)\|z_t\|^2+\delta\|\nabla z\|^2+\Big(f(u^1)-f(u^2),(-\Delta)^{-1}z_t+\delta z\Big)=I_1+I_2,
\end{equation}
where $\xi_z=(z,z_t)$,
\begin{equation*}
  \begin{split}
     & \Phi(\xi_z)=\frac{1}{2}\Big(\|z_t\|^2_{H^{-1}}+\|z\|^2+\|\nabla z\|^2\Big)+\delta(z_t,z)\sim \|(z,z_t)\|^2_{\h_{-1}},\\
      & I_1=\Big[(\e_1-\e_2)\|\nabla u^1\|^2+\e_2(u^1+u^2, -\Delta z)\Big]\Big(\Delta u^1,(-\Delta)^{-1}z_t+\delta z \Big),\\
    &    I_2=\e_2\|\nabla u^2\|^2\Big(\Delta z,(-\Delta)^{-1}z_t+\delta z \Big)
  \end{split}
\end{equation*}
for $\delta>0$ suitably small.  Obviously,
\[ |I_1+I_2|\leq C_K\|(z,z_t)\|^2_{\h_{-1}}+|\e_1-\e_2|^2,\]
   where we have used estimate \eqref{5.4}. By Assumption  \eqref{1.3}, the Sobolev embedding  $H^{2-\theta}\hookrightarrow L^{p+1}$ for $0<\theta\ll 1$ and the interpolation, we have
\begin{equation*}
  \begin{split}
     & (f(u^1)-f(u^2),z)\geq -C\|z\|^2+C\int_\lo\big(|u^1|^{p-1}+|u^2|^{p-1}\big)|z|^2dx, \\
      &  |(f(u^1)-f(u^2),(-\Delta)^{-1}z_t)|\leq C\int_\lo(1+|u^1|^{p-1}+|u^2|^{p-1})|z||(-\Delta)^{-1}z_t|dx\\
      &\leq  \frac{\delta C}{2}\int_\lo\big(1+|u^1|^{p-1}+|u^2|^{p-1}\big)|z|^2dx+C(1+\|u^1\|_{p+1}^{p-1}+\|u^2\|_{p+1}^{p-1}\big) \|(-\Delta)^{-1}z_t \|^2_{p+1}\\
         &\leq \frac{\delta C}{2}\int_\lo\big(|u^1|^{p-1}+|u^2|^{p-1}\big)|z|^2dx+\delta\|z_t\|^2+C_{K}(\|z\|^2+\|z_t\|^2_{H^{-1}}).
  \end{split}
\end{equation*}
Inserting above estimates into \eqref{5.6} yields
\begin{equation}\label{5.7}
\frac{\d}{\d t}\Phi(\xi_z(t))\leq |\e_1-\e_2|^2+C_K\Phi(\xi_z(t)),\ \ t>\tau.
\end{equation}
 Applying the Gronwall inequality to \eqref{5.7} over $(\tau,t)$  gives \eqref{5.3}.
\end{proof}

\begin{lemma}\label{53} Under the assumptions of Theorem \ref{51},  the family of processes $\{U^\e_g(t,\tau)\},g\in\ls, \e\in[0,1]$ has a uniformly ($w. r. t.\ g\in\ls$ and $\e\in[0,1]$) absorbing set $\b_0$, which  is bounded in $(H_0^1\cap L^{p+1})\times H_0^1$.
\end{lemma}
\begin{proof}
By Lemma \ref{42}, there exists a $T>0$ such that
\begin{equation*}
  \bigcup_{\e\in[0,1]}\bigcup_{g\in\ls}U^\e_g(t,0)\b\subset \b,\ \ t\geq T.
\end{equation*}
Let
\begin{equation}\label{5.20}
  \b_0=\bigcup_{\e\in[0,1]}\bigcup_{g\in\ls}\bigcup_{t\geq T+1}U^\e_g(t,0)\b (\subset \b).
\end{equation}
Then $\b_0$ is  the desired absorbing set.  Indeed, for any bounded set $D\subset \h$, there exists a  $t_D\geq0$ such that
\[\bigcup_{\e\in[0,1]}\bigcup_{g\in\ls}U^\e_g(t,0)D\subset \b\ \ \hbox{as}\ \ t\geq t_D.\]
When $t\geq t_D+T+1+\tau$, by Lemma \ref{29}, there exist at least one $g'\in \ls$ such that
\begin{equation*}
  \begin{split}
    U^\e_g(t,\tau)D & =U^\e_{g'}(t-\tau,0)D= U^\e_{g'}(t-\tau,t_D)U^\e_{g'}(t_D,0)D\\
      &  \subset U^\e_{g'}(t-\tau,t_D)\b= U^\e_{T(t_D)g'}(t-\tau-t_D,0)\b\subset \b_0
  \end{split}
\end{equation*}
for any $\e\in[0,1], g\in\ls, \tau\in\r$,  where we have used   translation identity \eqref{3.8}.
Due to
\begin{equation*}
  \b_0=\bigcup_{\e\in[0,1]}\bigcup_{g\in\ls}\bigcup_{t\geq T+1}U^\e_g(t,t-1)U^\e_g(t-1,0)\b\subset\bigcup_{\e\in[0,1]}\bigcup_{g\in\ls}\bigcup_{t\in\r}U^\e_g(t,t-1)\b,
\end{equation*}
we infer from estimate \eqref{5.19}  that $\b_0$ is bounded in $(H_0^1\cap L^{p+1})\times H_0^1$.
\end{proof}

\begin{proof}[Proof of Theorem \ref{51}]   If the formula  \eqref{5.1} does not hold.  There must exist $\delta>0, \e_0\in[0,1], \{\e_n\}\subset[0,1]$ with $\e_n\rightarrow \e_0$, and $\xi_n\in \mathcal{A}^{\e_n}_\ls$ such that
\begin{equation}\label{5.10}
 \mathrm{dist}_{\h_{-1}}\{\xi_n, \mathcal{A}_{\ls}^{\e_0}\}>\delta, \ \ \forall n.
\end{equation}
Due to $\h\hookrightarrow \h_{-1}$, we have
\begin{equation*}
  \sup_{g\in\ls}\mathrm{dist}_{\h_{-1}}\{U^{\e_0}_g(t,0)\b, \mathcal{A}_{\ls}^{\e_0} \}\leq C\sup_{g\in\ls}\mathrm{dist}_{\h}\{U^{\e_0}_g(t,0)\b, \mathcal{A}_{\ls}^{\e_0} \}\rightarrow 0,
\end{equation*}
which implies that there exists a $T>0$ such that when $t\geq T$, 
\begin{equation}\label{5.11}
\sup_{g\in\ls}\mathrm{dist}_{\h_{-1}}\{U^{\e_0}_g(t,0)\b, \mathcal{A}_{\ls}^{\e_0} \}\leq \frac{\delta}{3}\ \ \hbox{and }\ \  \bigcup_{\e\in[0,1]}\bigcup_{g\in\ls}U^\e_g(t,0)\b\subset \b.
\end{equation} 
Due to  $\xi_n\in \mathcal{A}_{\ls}^{\e_n}=\omega^{\e_n}_{0,\ls}(\b)$, there exist $g_n\in\ls$, $\eta_n\in\b$ and $t_n\geq 2T$ such that
\begin{equation}\label{5.13}
\|U^{\e_n}_{g_n}(t_n, 0)\eta_n-\xi_n\|_{\h_{-1}}\leq C \|U^{\e_n}_{g_n}(t_n, 0)\eta_n-\xi_n\|_{\h}\leq \frac{\delta}{3}, \ \ \forall n.
\end{equation}
Since
\begin{equation*}
  \begin{split}
    U^{\e_n}_{g_n}(t_n, 0)\eta_n & =U^{\e_n}_{g_n}(t_n, t_n-T)U^{\e_n}_{g_n}(t_n-T, 0)\eta_n \\
      &  =U^{\e_n}_{g_n}(t_n, t_n-T)z_n=U^{\e_n}_{T(t_n-T)g_n}(T, 0)z_n, \ \ \forall n,
  \end{split}
\end{equation*}
where $z_n=U^{\e_n}_{g_n}(t_n-T, 0)\eta_n\in \b$ for $t_n-T\geq T$,  we infer  from Lemma \ref{52}  that there exists  a $N>0$ such that
\begin{equation}\label{5.14}
  \|U^{\e_n}_{T(t_n-T)g_n}(T, 0)z_n-U^{\e_0}_{T(t_n-T)g_n}(T, 0)z_n \|_{\h_{-1}}\leq C_K |\e_n-\e_0|\leq \frac{\delta}{3}\ \ \hbox{as}\ \  n\geq N
\end{equation}
for   $\e_n\rightarrow \e_0$. Therefore, it follows  from estimates \eqref{5.11}-\eqref{5.14} that
\begin{equation*}
  \begin{split}&
    \mathrm{dist}_{\h_{-1}}\{\xi_n, \mathcal{A}_{\ls}^{\e_0}\}\\
     & \leq \|\xi_n-U^{\e_n}_{g_n}(t_n, 0)\eta_n \|_{\h_{-1}}  + \|U^{\e_n}_{T(t_n-T)g_n}(T, 0)z_n-U^{\e_0}_{T(t_n-T)g_n}(T, 0)z_n \|_{\h_{-1}}\\
     &\ \ \ \ +\mathrm{dist}_{\h_{-1}}\{U^{\e_0}_{T(t_n-T)g_n}(T, 0)z_n, \mathcal{A}_{\ls}^{\e_0}\}\leq \delta, \ \ \forall n\geq N,
  \end{split}
\end{equation*}
which violates \eqref{5.10}. Therefore,   formula  \eqref{5.1} holds.
\medskip

Now, we give the proof of formula \eqref{5.2} by contradiction. If formula \eqref{5.2} does not hold, there must exist  $s_0\in \r, g\in\ls, \delta>0, \e_0\in[0,1]$, sequences $\{\e_n\}\subset [0,1]$ with $\e_n\rightarrow \e_0$ and $\xi_n\in \K^{\e_n}_g(s_0)$ such that
\begin{equation}\label{5.15}
  \mathrm{dist}_{\h_{-1}}\{\xi_n, \K^{\e_0}_g(s_0)\}>\delta, \ \ \forall n.
\end{equation}
On the other hand,  the process $U^{\e_n}_g(t,\tau)$ has   a bounded full trajectory $\gamma_n=\{\xi_u^n(t)|t\in\r\}$ for each $n$  such that
\begin{equation}\label{5.16}
  \xi_n=\xi_u^n(s_0)\ \ \hbox{and}\ \ U^{\e_n}_g(t,\tau)\xi_u^n(\tau)=\xi_u^n(t),\ \ \forall t\geq \tau, \tau\in\r.
\end{equation}
Formula \eqref{4.6} shows that  $\xi_u^n(s)\in \K^{\e_n}_g(s)\subset \mathcal{A}_\ls^{\e_n}, \ \forall s\in\r$. By formula \eqref{5.1} and the compactness of $\mathcal{A}_\ls^{\e_0}$ in $\h_{-1}$,   there must  exist  a $\xi_u(s)\in \mathcal{A}_\ls^{\e_0}$ such that (subsequence if necessary),
\begin{equation}\label{5.17}
  \xi^n_u(s)\rightarrow \xi_u(s)\ \ \hbox{in}\ \ \h_{-1}, \ \ \forall s\in\r.
\end{equation}
Then we infer form Lemma \ref{52} that
\begin{align}\label{5.18}&
 \|U^{\e_n}_g(t,\tau)\xi_u^n(\tau)-U^{\e_0}_g(t,\tau)\xi_u(\tau)\|_{\h_{-1}}\nonumber\\
 &\leq C_K(\|\xi_u^n(\tau)-\xi_u(\tau)\|_{\h_{-1}}+|\e_n-\e_0|)
 \rightarrow 0 \ \ \hbox{as}\ \ n\rightarrow\infty, \ \forall t\geq \tau, \tau\in\r.
\end{align}
By the uniqueness of the limit,
\begin{equation*}
  \xi_u(t)=U^{\e_0}_g(t,\tau)\xi_u(\tau), \ \ \forall t\geq \tau, \tau\in\r,
\end{equation*}
which means   $\gamma=\{\xi_u(t)|t\in\r\}\in \K^{\e_0}_g$ and $\xi_u(s_0)\in \K^{\e_0}_g(s_0)$. Hence,
\begin{equation*}
  \mathrm{dist}_{\h_{-1}}\{\xi_n, \K^{\e_0}_g(s_0)\}\leq \mathrm{dist}_{\h_{-1}}\{\xi_u^n(s_0), \xi_u(s_0)\}\rightarrow 0, \ \ n\rightarrow \infty,
\end{equation*}
which violates \eqref{5.15}. Therefore, formula \eqref{5.2} holds.
\end{proof}

We consider the bounded uniformly absorbing set $\b_0$ as a topology space equipped with the partially strong topology as shown in \eqref{1.5}. Since $\b_0$ is bounded in $(H_0^1\cap L^{p+1})\times H_0^1$, this topology can be defined by the following metric $\rho$:
\begin{equation}\label{5.21}
  \rho(\xi_u,\xi_v)=\|\nabla (u_0-v_0)\|+\|u_1-v_1\|+\sum_{n=1}^{\infty}2^{-n}\frac{|(u_0-v_0, g_n)|}{1+|(u_0-v_0, g_n)|},
\end{equation}
where $\xi_u=(u_0,u_1), \xi_v=(v_0,v_1)\in\b_0$, $\{g_n\}\subset H^{-1}\cap L^{1+\frac{1}{p}}$ such that $\|g_n\|_{H^{-1}}=1$ and $span\{g_n| n\in\n\}$ is dense in $L^{1+\frac{1}{p}}$ (cf. \cite{Chueshov}).

\begin{corollary}\label{54} Let Assumption  \ref{11} be valid, with $g\in\ls$. Then
\begin{enumerate}[$(i)$]
  \item  the compact uniform attractors $\mathcal{A}_\ls^\e$ as shown in Theorem \ref{43} is upper semicontinuous at the point $\e_0\in [0,1]$ in the sense of partially strong topology, i.e.,
\begin{equation*}
 \lim_{\e\rightarrow \e_0}\mathrm{dist}_{\rho}\{\mathcal{A}_\ls^\e,\mathcal{A}_\ls^{\e_0}\}=0,
\end{equation*}
where
\begin{equation*}
  \mathrm{dist}_{\rho}\{A,B\}=\sup_{x\in A}\inf_{y\in B}\rho(x,y), \ \ A, B\subset \b_0;
\end{equation*}
    \item for any fixed $g\in \ls$ and $\e\in\ls$, the family of all kernel sections $\mathcal{A}^\e_g=\{\K_g^\e(t)\}_{t\in \mathbb{R}}$ is the pullback attractor of the process  $\{U^\e_g(t,\tau)\}$, and it is upper semicontinuous
    at the point $\e_0\in [0,1]$ in the sense of partially strong topology, i.e.,
       \begin{equation*}
 \lim_{\e\rightarrow \e_0}\mathrm{dist}_{\rho}\{\K_g^\e(s), \K_g^{\e_0}(s)\}=0,\ \ \forall   s\in \r.
\end{equation*}
\end{enumerate}
\end{corollary}

\begin{proof}  Since $\mathcal{A}_\ls^\e$ is the compact uniform attractor of the family of processes $\{U^\e_g(t,\tau\}, g\in \ls$ and \eqref{4.6} holds,  by the standard theory on the uniform attractor (cf. Chapter IV in \cite{C-V}), for any fixed $g\in \ls$ and $\e\in\ls$,
the family of all kernel sections $\mathcal{A}^\e_g=\{\K_g^\e(t)\}_{t\in \mathbb{R}}$ is just a pullback attractor of the process $\{U^\e_g(t,\tau\}$.

Due to
\begin{equation*}
  \frac{|(u_0-v_0, g_n)|}{1+|(u_0-v_0, g_n)|}\leq \frac{\|\nabla(u_0-v_0)\|\| g_n\|_{H^{-1}}}{1+|(u_0-v_0, g_n)|}\leq \|\nabla(u_0-v_0)\|, \ \ \forall n,
\end{equation*}
we see from \eqref{5.21} that
\begin{equation}\label{5.22}
\rho(\xi_u,\xi_v)\leq 2\|\xi_u-\xi_v\|_{H^1_0\times L^2}.
\end{equation}
 For any $\xi_u=(u_0,u_1), \xi_v=(v_0,v_1)\in\b_0$, by the interpolation,
\begin{equation}\label{5.23}
  \begin{split}
  \|\xi_u-\xi_v\|_{H^1_0\times L^2}   & \leq \|\nabla(u_0-v_0)\|+\|\nabla(u_1-v_1)\|^{\frac{1}{2}}\|u_1-v_1\|_{H^{-1}}^{\frac{1}{2}}\\
      & \leq C\|\xi_u-\xi_v\|_{\h_{-1}}^{\frac{1}{2}}.
  \end{split}
\end{equation}
Taking account of  $\mathcal{A}^\e_\ls\subset \b_0$ for all $\e\in[0,1]$, we infer from \eqref{5.22}-\eqref{5.23} and Theorem \ref{51} that
\begin{align*}&
\mathrm{dist}_{\rho}\{\mathcal{A}_\ls^\e,\mathcal{A}_\ls^{\e_0}\}\leq C[\mathrm{dist}_{\h_{-1}}\{\mathcal{A}_\ls^\e,\mathcal{A}_\ls^{\e_0}\}]^{\frac{1}{2}}\rightarrow 0,\\
&\mathrm{dist}_{\rho}\{\K_g^\e(s), \K_g^{\e_0}(s)\}\leq C[\mathrm{dist}_{\h_{-1}}\{\K_g^\e(s), \K_g^{\e_0}(s)\}]^\frac{1}{2}\rightarrow 0\ \ \hbox{as }\ \ \epsilon\rightarrow \e_0,
\ \ \forall   s\in \r.
\end{align*}
\end{proof}


\begin {thebibliography}{90} {\footnotesize
\bibitem{Bae}
J. J.  Bae, M. Nakao, Existence problem of global solutions of the Kirchhoff type wave equations with a localized
weakly nonlinear dissipation in exterior domains, Discrete Contin. Dyn. Syst.  11 (2004) 731-743.

\bibitem{Ball}
J. M. Ball, Global attractors for damped semilinear wave equations, Discrete Contin. Dyn. Syst.  10 (2004), 31-52.

\bibitem{Cavalcanti}
M. M. Cavalcanti, V. N. D. Cavalcanti, J. S. P. Filho, J. A. Soriano, Existence and exponential decay for a Kirchhoff-
Carrier model with viscosity, J. Math. Anal. Appl. 226 (1998) 40-60.


\bibitem{C-V}
V. V. Chepyzhov, M. I. Vishik, Attractors for  equations of mathematical physics, Amer. Math. Soc.,  Colloquium Publications, Vol. 49, Providence, RI, 2002.

\bibitem{Chueshov1}
 I. Chueshov, I. Lasiecka, Long-Time Behavior of Second Order Evolution Equations with Nonlinear Damping, Memoirs of
AMS 912, Amer. Math. Soc., Providence, 2008.

\bibitem{Chueshov}
I. Chueshov, Long-time dynamics of Kirchhoff wave models with strong
nonlinear damping, J. Differential Equations, 252 (2012) 1229-1262.

\bibitem{DYL}
P. Y. Ding,  Z. J. Yang, Y. N. Li,  Global attractor of the  Kirchhoff wave models with  strong nonlinear damping, Appl. Math. Lett. 76 (2018) 40-45.


\bibitem{zhou}
X. Fan, S. Zhou, Kernel sections for non-autonomous strongly damped wave equations of non-degenerate Kirchhoff-type, Appl. Math. Comput. 158 (2004) 253-266.

\bibitem{MM}
M. M. Freitas, P. Kalita, J. A. Langa, Continuity of non-autonomous attractors for hyperbolic perturbation of parabolic equations, J. Differential Equations, 264 (2018) 1886-1945.



\bibitem{Kalantarov}
 V. Kalantarov, S. Zelik, Finite-dimensional attractors for the quasi-linear strongly-damped wave equation, J. Differential
Equations, 247 (2009) 1120-1155.



\bibitem{Kirchhoff}
 G. Kirchhoff,
 Vorlesungen \"{u}ber Mechanik,
    Lectures on Mechanics, Teubner, Stuttgart, 1883.

\bibitem{Lu1} S. S. Lu, H. Q. Wu,  C. K. Zhong,
 Attractors for non-autonomous $2D$ Navier-Stokes equations with normal external forces,
Discrete Contin. Dyn. Syst.  13 (2005) 701-719.

\bibitem{Lu2}
S. S. Lu, Attractors for non-autonomous $2D$ Navier-Stokes equations with less regular normal forces, J. Differential Equations, 230 (2006) 196-212.

\bibitem{Lu3}
S. S. Lu, Attractors for non-autonomous reaction-diffusion systems with symbols without strong translation compactness, Asymptot. Anal.  54 (2007) 197-210.


\bibitem{M-Z}
H. L. Ma, C. K. Zhong, Attractors for the Kirchhoff equations with strong nonlinear damping, Appl. Math. Lett. 74 (2017) 127-133.


\bibitem{Ma2}
S. Ma, C. K. Zhong, H. Li, Attractors for non-autonomous wave equations with a new class of external forces,
J. Math. Anal. Appl.  337 (2008) 808-820.

\bibitem{Ma3}
S. Ma, C. K. Zhong, The attractors for weakly damped non-autonomous hyperbolic equations with a new class of external forces,  Discrete Contin. Dyn. Syst.  18 (2007) 53-70.

\bibitem{Matsuyama}
T. Matsuyama, R. lkehata, On global solution and energy decay for the wave equation of Kirchhoff-type with nonlinear damping term, J. Math. Anal. Appl.   204 (1996) 729-753.


\bibitem{Moise}
I. Moise, R. Rosa, X. Wang, Attractors for noncompact non-autonomous systems via energy equations, Discrete Contin. Dyn. Syst.   10 (2004) 473-496.

\bibitem{Nakao1}
M. Nakao, An attractor for a nonlinear dissipative wave equation of Kirchhoff type, J. Math. Anal. Appl.,  353  (2009) 652-659.

\bibitem{Nakao2}
M. Nakao, Z. J. Yang, Global attractors for some quasi-linear wave equations with a strong dissipation, Adv. Math. Sci. Appl.  17 (2007) 89-105.


\bibitem{Nishihara}
K.  Nishihara,  Decay properties of solutions of some quasilinear hyperbolic equations with strong
damping, Nonlinear Anal.  21  (1993) 17-21.


\bibitem{Ono}
K. Ono, Global existence, decay, and blowup of solutions for
some mildly degenerate nonlinear Kirchhoff strings,  J. Differential Equations, 137 (1997) 273-301.

\bibitem{Ono1}
K. Ono, On global existence, asymptotic stability
and blowing up of solutions for some degenerate
non-linear wave equations of Kirchhoff type
with a strong dissipation, Math. Methods  Appl. Sci.  20 (1997)  151-177.



\bibitem{Simon}
J. Simon,   Compact sets in the space $L^p(0,T;B)$,  Ann. Mat. Pura Appl.   146 (1986) 65-96.

\bibitem{Sun1}
C. Y. Sun, D. M. Cao, J. Q. Duan, Non-autonomous dynamics of wave equations with nonlinear damping and critical
nonlinearity, Nonlinearity, 19 (2006) 2645-2665.


\bibitem{Sun}
C. Y. Sun, D. M. Cao, J. Q. Duan, Uniform attractors for non-autonomous wave equations with nonlinear damping, SIAM J. Appl. Dyn. Syst.  6 (2007) 293-318.

\bibitem{BX}
B. X. Wang, Uniform attractors of non-autonomous discrete reaction-diffusion systems in weighted spaces, Int. J. Bifurcation Chaos, 18 (2008) 659-716.


\bibitem{zhong}
Y. H.  Wang, C. K. Zhong, Upper semicontinuity of pullback attractors for non-autonomous Kirchhoff wave  models, Discrete Contin. Dyn. Syst.  7 (2013) 3189-3209.


\bibitem{Y1}
  Z. J. Yang, Long-time behavior of the Kirchhoff type equation with strong damping in $\mathbb{R}^N$, J. Differential Equations, 242
(2007) 269-286.

\bibitem{Y-W}
Z. J.  Yang, Y. Q. Wang, Global attractor for the Kirchhoff type equation with
a strong dissipation, J. Differential Equations, 249 (2010) 3258-3278.

\bibitem{Y-D}
Z. J.  Yang, P. Y. Ding, Longtime dynamics of the Kirchhoff equation with strong
damping and critical nonlinearity on $\mathbb{R}^N$, J. Math. Anal. Appl. 434 (2016) 1826-1851.


\bibitem{Zelik2}
S. Zelik, Asymptotic regularity of solutions of a non-autonomous damped wave equation with a critical growth exponent, Comm. Pure Appl. Anal.  4 (2004) 921-934.


\bibitem{Zelik}
S. Zelik, Strong uniform attractors for non-autonomous dissipative PDEs with non translation-compact external forces,  Discrete Contin. Dyn. Syst.: B  20 (2015) 781-810.

}


\end{thebibliography}
\end{document}